\documentclass[%
a4paper,%
arxiv,%
defaults,%
]{myclass}

\usepackage[matrix,arrow]{xy}

\usepackage[%
bb,%
geom,%
]%
{mymacros}

\newcommand{\lb}{[}
\newcommand{\rb}{]}
\newcommand{\lctvs}{LCTVS\xspace}
\newcommand{\ipc}{\ip{\cdot}{\cdot}}
\newcommand{\dirac}{\partial\hspace*{-1.2ex}/}
\newcommand{\res}{\text{res}}
\newcommand{\pol}{\text{pol}}

\theoremstyle{\myrmkstyle}
\newtheorem{examples}[theorem]{Examples}

\mytitle{%
  How to Construct a Dirac Operator in Infinite Dimensions%
}

\mydate{\today}
\myauthor{%
  Andrew Stacey%
}

\myaddress{%
  Insitutt for matematiske fag\\
  NTNU\\
  7491 Trondheim\\
  Norway}

\myemail{%
  stacey@math.ntnu.no}
\myurl{%
  http://www.math.ntnu.no/~stacey}

\myabstract{%
We define the notion of a \emph{co\hyp{}Riemannian structure} and show how it can be used to define the Dirac operator on an appropriate infinite dimensional manifold.
In particular, this approach works for the smooth loop space of a so\hyp{}called \emph{string} manifold.
}

\mybibliographystyle

\begin{document}

\mymaketitle

\section{Introduction}

We introduce the notion of a \emph{co\hyp{}Riemannian structure}.
In infinite dimensions, to give a Riemannian structure is to give a Hilbert completion of each tangent space.
That is, to give an injective continuous linear map from each tangent space to a Hilbert space with dense image.
A \emph{co\hyp{}Riemannian structure} reverses this.
It consists of an injective continuous linear map from a Hilbert space to each tangent space with dense image.

The purpose of this definition is to provide a framework to allow more geometrical constructions than are possible with only a Riemannian structure.
In finite dimensional Riemannian geometry, almost the first result stated is the identification of the tangent and cotangent bundles via the inner products.
This identification is used throughout geometry, often implicitly, and is crucial in many geometrical constructions.
In contrast, a Riemannian structure on an infinite dimensional manifold may not define such an identification.
If it does it is called a \emph{strong} Riemannian structure; such can only exist on Hilbert manifolds and only if the Riemannian structure is chosen carefully.
Although many infinite dimensional manifolds have Hilbertian approximations, there are important examples that do not\emhyp{}such as diffeomorphism groups.
Moreover, even when the manifold is Hilbertian, the ``obvious'' Riemannian structure may not be strong\emhyp{}the space of \(L^{1,2}\)\enhyp{}loops is Hilbertian but the usual Riemannian structure is the \(L^2\)\enhyp{}inner product which is not strong.

A weak (i.e.\ not necessarily strong) Riemannian structure still defines an injective linear map from the tangent to cotangent spaces.
Thus any geometrical construction which only uses this aspect of the above identification can still be carried out for weak Riemannian structures.
There are many, however, that do not; for example, the Koszul formula for the Levi\hyp{}Civita connection, the Dirac operator on a spin manifold, and the gradient of a function.
The purpose of a co\hyp{}Riemannian structure is to provide an injective linear map from the cotangent to the tangent spaces meaning that any geometrical construction which only uses \emph{this} aspect of the identification can still be carried out.
For example, the Dirac operator and the gradient of a function.

Having both a Riemannian and co\hyp{}Riemannian structure will further permit any geometrical construction which uses both maps but which does not require them to be mutually inverse.

\medskip

Our motivating application of this theory is the construction of the Dirac operator in infinite dimensions, in particular on loop spaces.
This problem goes back at least to Witten~\cite{ew} where Witten formally applied the equivariant Atiyah\enhyp{}Singer index theorem to the hypothetical Dirac operator on a suitable loop space.
The resulting object now goes by the name of the \emph{Witten genus} and has provided much interesting mathematics.
Although the Witten genus itself can be defined, its original construction has not been made rigorous.
This has led to its being considered mainly as a topological object without the usual flow of ideas and proofs from analysis that index theory usually affords.
There have been various attempts since then to construct such an operator, usually by assuming some conditions on the original manifold or by reinterpreting the notion of ``loop space'' in some fashion; see \cite{ct2,jjrl2,ew4,rl8,rl7,rl6,rl5,rl4,mstw}.

The problem with defining a Dirac operator in infinite dimensions is simple to describe.
One can gather all the required ingredients as in finite dimensions and set out on the construction, mirroring that of finite dimensions.
However, at one stage a certain map is required.
There is an obvious map in the opposite direction which in finite dimensions is invertible.
In infinite dimensions it is not and the construction fails.

One of the ingredients for the Dirac operator is a Riemannian structure on the manifold.
What we shall show is that if this is replaced by a co\hyp{}Riemannian structure then the ``obvious map'' referred to above is now in the right direction and the construction succeeds.

The final step is to construct an appropriate co\hyp{}Riemannian structure on the free loop space of a string manifold.
This is more complicated and we defer it to a companion paper \cite{math/0809.3108}.

\medskip

Let us now give a more detailed overview of this paper.
In section~\ref{sec:coriem} we shall define and classify co\hyp{}Riemannian structures.
In fact, there is nothing that is particular to the tangent bundle so we shall define and classify co\hyp{}orthogonal structures.
Also, our classification applies equally to orthogonal structures, and refines the ``weak\enhyp{}strong'' classification so we also classify orthogonal structures.
Moreover, we shall start with the orthogonal structures as these are more familiar and then ``op'' (almost) everything to produce the co\hyp{}orthogonal structures.

The definition of a co\hyp{}orthogonal structure is essentially as given in the first paragraph of this introduction.

The classification scheme is based on the question ``How like a bundle of Hilbert spaces is the resulting object?''.
The standard ``weak\enhyp{}strong'' classification can be regarded as a ``not exactly'' or ``exactly'' answer (though the term ``weak orthogonal structure'' can mean either ``not strong'' or ``not necessarily strong'').
To refine this we use the fact that an inner product on a locally convex topological vector space defines a Hilbert completion and so to an orthogonal structure on a vector bundle we can assign a family of Hilbert spaces.
We then apply the above question to this family and to the map that includes the original bundle in this family of Hilbert spaces.
More specifically, we classify the amount of this structure that can be simultaneously locally trivialised.
Depending as it does on the Hilbert completions this classification applies equally well to orthogonal and co\hyp{}orthogonal structures.

In passing we mention the r\^ole of the structure groups.
In infinite dimensions it is relatively rare that one specifies a vector bundle using the full general linear group as its structure group.
This means that structure bundles are more prominent in infinite dimensions and that there is more flexibility in how to modify them\emhyp{}one may wish to be able to enlarge the structure group as well as reduce it.

In section~\ref{sec:dirac} we shall show how to construct a Dirac operator in infinite dimensions using an appropriate co\hyp{}Riemannian structure.
The construction follows that of finite dimensions with no changes: once one has the right starting place, everything else is straightforward.
We shall therefore not go into further detail here as there is not much difference between giving an overview of the construction and the actual construction itself.
Suffice it to say that, as remarked earlier, starting with the co\hyp{}Riemannian structure means that the crucial map goes in the correct direction to define the operator.

\section{Orthogonal Structures in Infinite Dimensions}
\label{sec:coriem}

In this section we shall define and classify orthogonal and co\hyp{}orthogonal structures on infinite dimensional vector bundles.

Let us outline some conventions.
All vector spaces will be locally convex topological vector spaces (\lctvs).
All linear maps and bilinear maps will be assumed to be continuous unless otherwise stated.
Our definitions make sense in both the topological and smooth categories and so we shall not specify which.
All maps, groups, representations, choices will be assumed to be valid in the appropriate category.

In the following, we fix a manifold \(X\), a \lctvs \(V\), a group \(G\), and a vector bundle \(\xi\) with fibre \(V\) and structure group \(G\).

\begin{defn}
An \emph{orthogonal structure} on \(\xi\) is a choice of inner product on the fibres of \(\xi\).
\end{defn}

\begin{remark}
\begin{enumerate}
\item To reiterate the convention made above, these choices must vary in the manner appropriate to the category.

\item The existence of an orthogonal structure is very weak indeed.
If the base manifold \(X\) admits partitions of unity and the fibre \(V\) admits inner products then \(\xi\) admits an orthogonal structure.
\end{enumerate}
\end{remark}

\begin{example}
The \lctvs \(\R^{\N}\) has the property that any map from it to a Banach space must factor through a projection to some \(\R^n\).
Hence it cannot admit an inner product and so any bundle with fibre \(\R^{\N}\) does not admit inner products.
\end{example}

An inner product, \(\ipc\), on \(V\) defines an associated Hilbert space, \((\overline{V}, \ipc)\), together with a continuous, injective, linear map \(V \to \overline{V}\) with dense image.
We refer to this as a \emph{Hilbert completion} of \(V\).
Thus specifying an orthogonal structure on \(\xi\) associates to each fibre a Hilbert completion.
However, these Hilbert spaces may not patch together nicely over the base space.

\begin{defn}
A \emph{locally equivalent orthogonal structure} on \(\xi\) is a smooth choice of inner product on the fibres of \(\xi\) such that \(\xi\) admits a local trivialisation with the property that the transition functions preserve the equivalence class of the inner product.
\end{defn}

\begin{remark}
\begin{enumerate}
\item Given that we have already specified the structure group of \(\xi\), when we talk of ``local trivialisations'' we obviously mean with respect to the given structure group.

\item Recall that two inner products on \(V\) are equivalent if their associated norms are equivalent.
If this is so then the identity on \(V\) extends to an isomorphism between the Hilbert completions (though not necessarily isometrically).

Therefore, a locally equivalent orthogonal structure means that there is a standard Hilbert completion \((\overline{V}, \ipc)\) of \(V\) and a local trivialisation of \(\xi\) in which the Hilbert completions of the fibres of \(\xi\) look like \(\overline{V}\).

\item For such a structure to exist we need to reduce the structure group of \(\xi\) to a subgroup whose action on \(V\) extends to an action on \(\overline{V}\) by isomorphisms.

\item When talking of orthogonal structures in finite dimensions it is usual to think only of reducing the structure group.
Often the initial structure group is the full general linear group in which case reduction is the only possible direction.
However, in infinite dimensions using the full general linear group is relatively uncommon.
It is therefore possible that one may be able to enlarge the structure group (preferably without changing its homotopy type, or without changing it too much).

The reason one might wish to do this is that if the vector bundle with its original structure group does not admit a locally equivalent orthogonal structure, one might be able to find an enlargement of the structure group so that it does.
\end{enumerate}
\end{remark}

\begin{example}
Here is a bundle which does not admit such a structure.
Let \(L \C \coloneqq \Ci(S^1,\C)\) and let \(L_0 \C\) be the subspace of loops which integrate to zero.
Let \(D \colon L_0 \C \to L_0 C\) be the isomorphism given by differentiation: \(D \gamma = \gamma'\).
Define a bundle \(\xi \to S^1\) by quotienting \(L_0 \C \times \lb 0,1 \rb\) by the relation \((\gamma,0) \simeq (D \gamma, 1)\).

We claim that \(\xi\) does not admit a locally equivalent orthogonal structure.
For it to do so we would need an inner product \(\ipc\) on \(L_0 \C\) which was equivalent to the inner product \(\ip{D \cdot}{D \cdot}\).
In particular, there must be a Banach completion of \(L_0 \C\) on which \(D\) acts as an isomorphism.
This is impossible.

However, the standard partition\hyp{}of\hyp{}unity argument shows that \(\xi\) admits an orthogonal structure.
\end{example}

In a locally equivalent orthogonal structure we can locally trivialise the completion but not necessarily the inner products.

\begin{defn}
A \emph{locally trivial orthogonal structure} on \(\xi\) is a smooth choice of inner product on the fibres of \(\xi\) such that \(\xi\) admits a local trivialisation with the property that the transition functions preserve the inner product.
\end{defn}

\begin{remark}
The extra piece here is that the inner product itself is locally constant, not just its equivalence class.
Thus for this to exist we add to the conditions for a locally equivalent orthogonal structure the condition that the subgroup act by isometries.
\end{remark}

\begin{example}
Let \(L_0 \C\) and \(D\) be as above.
With respect to the splitting \(L \C = L_0 \C \oplus \C\) define an inner product \(\ipc_{-1}\) as
\[
  \ip{\alpha + \lambda}{\beta + \mu}_{-1} \coloneqq \ip{D^{-1}\alpha + \lambda}{D^{-1}\beta + \mu}
\]
where \(\ipc\) is the usual \(L^2\)\enhyp{}inner product on \(L \C\).
Write the Hilbert completion as \(L^{-1,2} \C\).
We claim that \(L S^1\) acts continuously on this Hilbert space but that the subgroup which acts by isometries is the constant loops, \(S^1\).
Let \(\zeta \colon S^1 \to S^1\) be the identity map.
Define a vector bundle \(\xi \to S^1\) by quotienting \(L \C \times \lb 0,1 \rb\) by the relation \((\gamma, 0) \simeq (\zeta \gamma, 1)\).
As multiplication by \(\zeta\) extends to \(L^{-1,2} \C\), this can be given a locally equivalent orthogonal structure.
However, there is no reduction to \(S^1\) so this does not admit a locally trivial orthogonal structure.
\end{example}

Thus at this last definition we know that we can view the inner product as coming from a specific inner product on a typical fibre, thus also a specific Hilbert completion.
However, although the class of Hilbert completions is (locally) constant, they may not patch together into a bundle themselves.
Thus we introduce the notion of a structure group for an orthogonal structure.
In so doing we shift focus slightly from inner products to Hilbert completions.

\begin{defn}
\label{def:sostrgrp}
Let \(K\) be a group with a given action on a Hilbert space (not necessarily by isometries).
An \emph{orthogonal structure with structure group \(K\)} on \(\xi\) consists of a bundle \(\zeta \to X\) of Hilbert spaces with structure group \(K\) together with a vector bundle map \(\xi \to \zeta\) which is injective on fibres.
\end{defn}

\begin{remark}
\begin{enumerate}
\item This defines an orthogonal structure on \(\xi\) by pulling back the inner products on the fibres of \(\zeta\) to those of \(\xi\).

\item We do not assume that the fibres of \(\xi\) are dense in those of \(\zeta\).
We do not need this assumption to define the orthogonal structure on \(\xi\) and not having it gives us a little more flexibility.
If we have an orthogonal structure on \(\xi\) it is entirely possible that the dimension of the Hilbert completions will vary over \(X\).
We can still hope for a structure group in this situation as we take the fibre large enough to accommodate all these completions.
\end{enumerate}
\end{remark}

\begin{defn}
\label{def:leostrgrp}
Let \(K\) be a group with a given action on a Hilbert space (not necessarily by isometries).
A \emph{locally equivalent orthogonal structure with structure group \(K\)} on \(\xi\) consists of a bundle \(\zeta \to X\) of Hilbert spaces with structure group \(K\) together with a vector bundle map \(\xi \to \zeta\) which is injective on fibres such that there is a simultaneous local trivialisation of \(\xi\) and \(\zeta\) which takes the map \(\xi \to \zeta\) to the inclusion of \(V\) in some fixed Hilbert completion.
\end{defn}

\begin{remark}
For this to happen we need to add to the conditions for a locally equivalent orthogonal structure the assumption that the action of the subgroup of \(G\) on the Hilbert completion \(\overline{V}\) factor through the group \(K\).
\end{remark}

If \(K\) acts by isometries, we obtain the following structure.

\begin{defn}
\label{def:ltostrgrp}
Let \(K\) be a group with a given action on a Hilbert space by isometries.
A \emph{locally trivial orthogonal structure with structure group \(K\)} on \(\xi\) consists of a bundle \(\zeta \to X\) of Hilbert spaces with structure group \(K\) together with a vector bundle map \(\xi \to \zeta\) which is injective on fibres such that there is a simultaneous local trivialisation of \(\xi\) and \(\zeta\) which takes the map \(\xi \to \zeta\) to the inclusion of \(V\) in some fixed Hilbert completion.
\end{defn}

\begin{remark}
For a locally equivalent orthogonal structure with structure group \(K\) the most common choices for \(K\) contain within them a homotopy equivalent subgroup that acts by isometries.
One can therefore trivialise the bundle of Hilbert spaces so that the transition functions act by isometries.
If one can \emph{simultaneously} trivialise the original bundle then one has a locally trivial orthogonal structure, otherwise just a locally equivalent one.
\end{remark}

The final level of classification involves the maps from \(\xi\) to the associated Hilbert bundle.
These are essentially linear definitions so we define them for vector spaces and make the obvious generalisation to vector bundles.

\begin{defn}
Let \(\m{J}(V,H) \subseteq \m{L}(V,H)\) be a functorial (in \(H\)) ideal of operators from \(V\) to Hilbert spaces.
We say that an inner product on \(V\) is of \emph{class \(\m{J}\)} if the inclusion of \(V\) into the associated Hilbert completion is of class \(\m{J}\).

Let \(\m{K}(H_1,H_2) \subseteq \m{L}(H_1,H_2)\) be an ideal of operators from Hilbert spaces to Hilbert spaces (functorial in both arguments).
We say that an inner product on \(V\) is of \emph{class per\hyp{}\m{J}} if whenever \(H_1 \to V\) is a continuous linear map then the composition \(H_1 \to V \to \overline{V}\) is of class \(\m{K}\), and the induced map \(\m{L}(H_1,V) \to \m{K}(H_1,\overline{V})\) is continuous.

We extend these to orthogonal structures in the obvious way.
\end{defn}

\begin{examples}
\begin{enumerate}
\item The inclusion of \(L^{1,2} S^1\) in \(L^2 S^1\) is a Hilbert\enhyp{}Schmidt operator.
Therefore the \(L^2\)\enhyp{}inner product on the space of \(L^{1,2}\)\enhyp{}loops in a smooth manifold is an orthogonal structure of Hilbert\enhyp{}Schmidt type.

\item If the typical fibre of the vector bundle is nuclear or dual\hyp{}nuclear then every orthogonal structure on it will be per\hyp{}nuclear.
\end{enumerate}
\end{examples}

We now wish to ``op'' all of this to define \emph{co\hyp{}orthogonal structures}.
The main difficulty is that there is not a suitable body of linear theory to fall\hyp{}back on.
This makes it unclear what exactly a ``co\hyp{}inner product'' should be\emhyp{}it is not even clear that this can be given a sensible definition mirroring all the properties of an inner product.
The situation becomes easier when working with the definitions involving structure groups as these hinge on maps from the \lctvs \(V\) to some Hilbert space.
Thus we reverse this and consider maps from Hilbert spaces into \(V\).

Let us give the basic definitions, corresponding to definitions~\ref{def:sostrgrp}, \ref{def:leostrgrp}, and \ref{def:ltostrgrp}.
The subsequent modifications are obvious.

\begin{defn}
Let \(K\) be a group with a given action on a Hilbert space (not necessarily by isometries).
A \emph{co\hyp{}orthogonal structure with structure group \(K\)} on \(\xi\) consists of a bundle \(\zeta \to X\) of Hilbert spaces with structure group \(K\) together with a vector bundle map \(\zeta \to \xi\) which is dense on fibres.
\end{defn}

\begin{defn}
Let \(K\) be a group with a given action on a Hilbert space (not necessarily by isometries).
A \emph{locally equivalent co\hyp{}orthogonal structure with structure group \(K\)} on \(\xi\) consists of a bundle \(\zeta \to X\) of Hilbert spaces with structure group \(K\) together with a vector bundle map \(\zeta \to \xi\) for which there is a simultaneous local trivialisation of \(\xi\) and \(\zeta\) which takes the map \(\xi \to \zeta\) to the inclusion of some fixed Hilbert space as a dense subspace of \(V\).
\end{defn}

If \(K\) acts by isometries, we obtain the following structure.

\begin{defn}
Let \(K\) be a group with a given action on a Hilbert space by isometries.
A \emph{locally trivial co\hyp{}orthogonal structure with structure group \(K\)} on \(\xi\) consists of a bundle \(\zeta \to X\) of Hilbert spaces with structure group \(K\) together with a vector bundle map \(\zeta \to \xi\) for which there is a simultaneous local trivialisation of \(\xi\) and \(\zeta\) which takes the map \(\xi \to \zeta\) to the inclusion of some fixed Hilbert space as a dense subspace of \(V\).
\end{defn}

The relationship between orthogonal and co\hyp{}orthogonal structures is captured in the following result.

\begin{proposition}
A co\hyp{}orthogonal structure on \(\xi\) defines an orthogonal structure on \(\xi^*\) of the same type and an injective bundle map \(\xi^* \to \xi\) (conjugate linear if over \C).

An orthogonal structure on \(\xi\) defines an injective bundle map \(\xi \to \xi^*\) (conjugate linear if over \C).
If \(V\), the model space of \(\xi\), is semi\hyp{}reflexive then an orthogonal structure on \(\xi\) defines a co\hyp{}orthogonal structure on \(\xi^*\) of the same type.
\end{proposition}

\begin{proof}
These are all linear properties.
The maps are simplest.
An inner product on \(V\) defines a map \(V \to V^*\) by \(v \mapsto (u \mapsto \ip{v}{u})\).
If \(H\) is the Hilbert completion, this is \(V \to H \cong H^* \to V^*\).
On the other hand, a map \(H \to V\) with \(H\) a Hilbert space defines a map \(V^* \to V\) via \(V^* \to H^* \cong H \to V\).

Now let us consider the case of a map \(T \colon U \to W\) of \lctvs with \(U\) semi\hyp{}reflexive, \(T\) injective, \(\im T\) dense.
This dualises to a map \(T^* \colon W^* \to U^*\).
This is injective since if \(T^* f = 0\) then \(f \restrict_{\im T} = 0\) and so \(f = 0\).
To show that it has dense image, let us assume for a contradiction that it does not.
Then there is some non\hyp{}zero \(g \in U^{* *}\) such that \(g \restrict_{\im T^*} = 0\).
That is, for all \(f \in W^*\), \(g (T^* f) = 0\).
As \(U\) is semi\hyp{}reflexive, there is some, necessarily non\hyp{}zero, \(u \in U\) such that \(g(h) = h(u)\) for all \(h \in U^*\).
In particular, for \(f \in W^*\), \(g (T^* f) = T^* f (u) = f(T u)\).
Hence \(f(T u) = 0\) contradicting the injectivity of \(T\).

Since Hilbert spaces are semi\hyp{}reflexive, we can apply this to arbitrary co\hyp{}orthogonal structures.
\end{proof}

\begin{examples}
\begin{enumerate}
\item All of our examples of orthogonal structures involved semi\hyp{}reflexive spaces and so can be dualised to co\hyp{}orthogonal structures.

\item The work of \cite{math/0809.3108} shows that the tangent space of the space of smooth loops in a Riemannian manifold, say \(M\), can be given a homotopy locally trivial orthogonal structure with structure group \(O_J\).
This relies on a reduction of the structure group from \(L \gl_n\) to \(L_\pol O_n\) to define a locally equivalent orthogonal structure.
To define a locally trivial orthogonal structure requires a modification of the structure group as outlined above.
There are many such structures; a simple one\hyp{}parameter family can be constructed as follows.
Let \(r \in (0,1)\).
Let \(L^2_r \R^n\) be the family of loops which have an analytic extension over the annulus of inner radius \(r\) and outer radius \(r^{-1}\) and are square integrable on the boundary.
Using the above reduction of the structure group, there is a Hilbert bundle over \(L M\) with a map to the tangent bundle which, in charts, looks like the obvious inclusion \(L^2_r \R^n \to L \R^n\).
\end{enumerate}
\end{examples}

\section{Constructing the Dirac Operator}
\label{sec:dirac}

In this section we shall explain how to use co\hyp{}orthogonal structures to define the Dirac operator on an infinite dimensional manifold.
We shall also follow the steps of the construction starting with an orthogonal structure to best illustrate the point where this fails and the co\hyp{}orthogonal one succeeds.

The theory of spin groups and their associated representations is developed in \cite{rppr} and \cite{apgs}.
In brief, let \(H\) be a real separable infinite dimensional Hilbert space.
There is a certain group, \(\gl_J\) (also called \(\gl_\res\)) acting on \(H\).
The subgroup of this which acts by isometries is written \(S O_J\) (or \(S O_\res\)) and this subgroup is homotopy equivalent to \(\gl_J\).
This subgroup has a central extension, \(\spin_J\) (also called \(\spin_\res\)), by \(S^1\).

\begin{defn}
Let \(X\) be a smooth manifold.
\begin{enumerate}
\item Let \(\zeta \to X\) be a bundle of Hilbert spaces over \(X\) with structure group \(S O_J\).
A \emph{spin structure} on \(\zeta\) consists of the following data.

\begin{enumerate}
\item A lift of the structure group of \(\zeta\) from \(S O_J\) to \(\spin_J\).

\item A connection on the principal \(\spin_J\)\enhyp{}bundle.
If a connection on \(\zeta\) has already been specified the connection on the \(\spin_J\)\enhyp{}bundle should be a lift of that on \(\zeta\).
\end{enumerate}

\item Let \(\xi \to X\) be a vector bundle.
A \emph{(co\hyp{})spin structure} on \(\xi\) consists of the following data.

\begin{enumerate}
\item A (co\hyp{})orthogonal structure on \(\xi\) with structure group \(\gl_J\).

\item A spin structure on the associated bundle of Hilbert spaces, where the orthogonal structure is used to reduce the structure group of the Hilbert spaces from \(\gl_J\) to \(S O_J\).
\end{enumerate}

\item We further classify (co\hyp{})spin structures according to the classification of their (co\hyp{})orthogonal structures.

\item A \emph{(co\hyp{})spin manifold} is a manifold together with a choice of (co\hyp{})spin structure on its tangent bundle.
\end{enumerate}
\end{defn}

The following is well\hyp{}known.

\begin{lemma}
Let \(X\) be a spin or co\hyp{}spin manifold.
Let \(\overline{T X}\) be the Hilbert bundle associated to the (co\hyp{})orthogonal structure on \(T X\).
There are complex bundles \(S^+, S^- \to X\) modelled on Hilbert spaces together with covariant differential operators \(\nabla^{\pm}\) on \(S^\pm\) and fibrewise bilinear maps \(c^\pm \colon \overline{T X} \times S^\pm \to S^\mp\).

The maps \(c^\pm\) define linear maps \(\overline{T X} \otimes_2 S^\pm \to S^\mp\) where \(\otimes_2\) is the Hilbert completion of the tensor product. \noproof
\end{lemma}

The Hilbert completion of the tensor product of two Hilbert spaces has the following property.
If \(H_1\) and \(H_2\) are Hilbert spaces then \(H_1^* \otimes_2 H_2\) is canonically isomorphic to the space of Hilbert\enhyp{}Schmidt operators \(H_1 \to H_2\); we shall denote this latter space by \(\m{H}(H_1, H_2)\).

\begin{theorem}
\label{th:dirac}
Let \(X\) be a per\enhyp{}Hilbert\enhyp{}Schmidt co\hyp{}spin manifold.
Then \(X\) admits a differential operator
\[
  \dirac^+ \colon \Gamma(S^+) \to \Gamma(S^-)
\]
which we call the \emph{Dirac operator}.
\end{theorem}

\begin{proof}
Let us proceed along the construction.
We wish to define a map from \(\Gamma(S^+)\) to \(\Gamma(S^-)\).
This map is a composition.
Its beginning and end are simple to describe.
The covariant differential operator on \(S^+\) is a map
\[
  \nabla^+ \colon \Gamma(S^+) \to \Gamma(\m{L}(T X, S^+)).
\]
The bilinear map extends to a map
\[
  c^+ \colon \Gamma(\overline{T X} \otimes_2 S^+) \to \Gamma(S^-).
\]
Recall that \(\overline{T X}\) is the bundle of Hilbert spaces associated to the co\hyp{}Hilbert structure on \(T X\).

As \(\overline{T X}\) is a bundle of real Hilbert spaces, it is canonically identified with its dual.
Thus we can regard the domain of \(c^+\) as \(\Gamma((\overline{T X})^* \otimes_2 S^+)\) and hence as \(\Gamma(\m{H}(\overline{T X}, S^+))\).
It remains to pass from \(\m{L}(T X, S^+)\) to \(\m{H}(\overline{T X}, S^+)\).

As we started from a co\hyp{}spin structure we have a map \(\overline{T X} \to T X\).
Thus we obtain \(\m{L}(T X, S^+) \to \m{L}(\overline{T X}, S^+)\).
Now we use the assumption that the co\hyp{}spin structure is of per\hyp{}Hilbert\enhyp{}Schmidt type to deduce that this factors through \(\m{H}(\overline{T X}, S^+)\).
Thus we define
\[
  \dirac^+ \colon \Gamma(S^+) \xrightarrow{\nabla^+} \Gamma(\m{L}(T X, S^+)) \to \Gamma(\m{H}(\overline{T X}, S^+)) \cong \Gamma(\overline{T X} \otimes_2 S^+) \xrightarrow{c^+} \Gamma(S^-). \qedhere
\]
\end{proof}

If we had started from a spin structure we would have been fine up to the point where we wanted a map from \(\Gamma(\m{L}(T X, S^+))\) to \(\Gamma(\m{H}(\overline{T X}, S^+))\).
As, for a spin structure, we have a map \(T X \to \overline{T X}\) we would obtain instead \(\m{H}(\overline{T X}, S^+) \to \m{L}(T X, S^+)\).
In finite dimensions this is an isomorphism and so we can still define the Dirac operator.
In infinite dimensions it cannot be so (the former is a Hilbert space, the latter is not).

Of course, under the conditions of Theorem~\ref{th:dirac}, \(X\) admits a Dirac operator \(\dirac^- \colon \Gamma(S^-) \to \Gamma(S^+)\) as well.

\medskip

There are two views of the construction of the Dirac operator in finite dimensions.
One views, as we did above, the bundles \(S^+\) and \(S^-\) as being constructed from the tangent spaces.
The other views them as being constructed from the cotangent spaces.
The two views are equivalent as the Riemannian structure identifies the tangent and cotangent spaces leading to natural isomorphisms at every stage of the construction.
The (formal) difference is that in one the map \(c^+\) has domain \(T M \times S^+\) whilst in the other it has domain \(T^* M \times S^+\).

We have the same two views in infinite dimensions.
The co\hyp{}orthogonal structure on \(T X\) defines an orthogonal structure on \(T^* X\).
The associated Hilbert bundle, \(\overline{T^* X}\), is dual to \(\overline{T X}\).
The inclusion \(T^* X \to \overline{T^* X}\) defines an inclusion \(T^* X \otimes S^+ \to \overline{T^* X} \otimes S^+\).
As the co\hyp{}orthogonal structure is per\hyp{}Hilbert\enhyp{}Schmidt, this extends to a map \(\m{L}(T X, S^+) \to \overline{T^* X} \otimes_2 S^+\).
Hence we define an operator
\[
  \Gamma(S^+) \xrightarrow{\nabla^+} \Gamma(\m{L}(T X, S^+)) \to \Gamma(\overline{T^* X} \otimes_2 S^+) \xrightarrow{c^+} \Gamma(S^-).
\]

Under the canonical identification of \(\overline{T^* X}\) with \(\overline{T X}\), these two operators are the same.

\medskip

Recall that a \emph{string manifold} is a finite dimensional manifold together with enough structure that its loop space is naturally a spin manifold.
In \cite{math/0809.3108} we prove the following theorem.

\begin{theorem}
Let \(M\) be a finite dimensional string manifold.
Then \(L M \coloneqq \Ci(S^1, M)\) is a locally equivalent co\hyp{}spin manifold of per\hyp{}Hilbert\enhyp{}Schmidt type.
Moreover, this structure is \(S^1\)\enhyp{}equivariant.
\end{theorem}

We therefore obtain the following corollary.

\begin{corollary}
Let \(M\) be a string manifold.
Then \(L M \coloneqq \Ci(S^1, M)\) admits a \(S^1\)\enhyp{}equivariant Dirac operator. \noproof
\end{corollary}

\mybibliography{arxiv,articles,books,misc}


\begin{thebibliography}{L{\'e}a02b}

\bibitem[JL97]{jjrl2}
J.~D.~S. Jones and R.~L{\'e}andre.
\newblock A stochastic approach to the {D}irac operator over the free loop
  space.
\newblock {\em Tr. Mat. Inst. Steklova}, 217(Prostran. Petel i Gruppy
  Diffeomorf.):258--287, 1997.

\bibitem[L{\'e}a01a]{rl8}
R{\'e}mi L{\'e}andre.
\newblock Quotient of a loop group and {W}itten genus.
\newblock {\em J. Math. Phys.}, 42(3):1364--1383, 2001.

\bibitem[L{\'e}a01b]{rl7}
R{\'e}mi L{\'e}andre.
\newblock A stochastic approach to the {E}uler-{P}oincare characteristic of a
  quotient of a loop group.
\newblock {\em Rev. Math. Phys.}, 13(10):1307--1322, 2001.

\bibitem[L{\'e}a02a]{rl6}
R{\'e}mi L{\'e}andre.
\newblock Analysis on loop spaces, and topology.
\newblock {\em Mat. Zametki}, 72(2):236--257, 2002.

\bibitem[L{\'e}a02b]{rl5}
R{\'e}mi L{\'e}andre.
\newblock Full stochastic {D}irac-{R}amond operator over the free loop space.
\newblock In {\em Proceedings of the {C}onference {D}edicated to the 90th
  {A}nniversary of {B}oris {V}ladimirovich {G}nedenko ({K}yiv, 2002)},
  volume~8, pages 269--280, 2002.

\bibitem[L{\'e}a02c]{rl4}
R{\'e}mi L{\'e}andre.
\newblock A stochastic approach to {W}itten's explanation of the rigidity
  theorem for a homogeneous manifold.
\newblock In {\em Ukrainian {M}athematics {C}ongress---2001 ({U}krainian)},
  pages 107--116. Nats{\=\i}onal. Akad. Nauk Ukra{\"\i}ni {\=I}nst. Mat., Kiev,
  2002.

\bibitem[PR94]{rppr}
R.~J. Plymen and P.~L. Robinson.
\newblock {\em Spinors in {H}ilbert space}, volume 114 of {\em Cambridge Tracts
  in Mathematics}.
\newblock Cambridge University Press, Cambridge, 1994.

\bibitem[PS86]{apgs}
Andrew Pressley and Graeme Segal.
\newblock {\em Loop groups}.
\newblock Oxford Mathematical Monographs. The Clarendon Press Oxford University
  Press, New York, 1986.
\newblock Oxford Science Publications.

\bibitem[Sta]{math/0809.3108}
Andrew Stacey.
\newblock {The co-Riemannian structure of smooth loop spaces}, arXiv:0809.3108
  [math].

\bibitem[SW03]{mstw}
Mauro Spera and Tilmann Wurzbacher.
\newblock The {D}irac-{R}amond operator on loops in flat space.
\newblock {\em J. Funct. Anal.}, 197(1):110--139, 2003.

\bibitem[Tau89]{ct2}
Clifford~Henry Taubes.
\newblock {$S\sp 1$} actions and elliptic genera.
\newblock {\em Comm. Math. Phys.}, 122(3):455--526, 1989.

\bibitem[Wit88]{ew}
Edward Witten.
\newblock The index of the {D}irac operator in loop space.
\newblock In {\em Elliptic curves and modular forms in algebraic topology
  (Princeton, NJ, 1986)}, volume 1326 of {\em Lecture Notes in Math.}, pages
  161--181. Springer, Berlin, 1988.

\bibitem[Wit99]{ew4}
Edward Witten.
\newblock Index of {D}irac operators.
\newblock In {\em Quantum fields and strings: a course for mathematicians,
  {V}ol. 1, 2 ({P}rinceton, {NJ}, 1996/1997)}, pages 475--511. Amer. Math.
  Soc., Providence, RI, 1999.

\end{thebibliography}
\end{document}